\documentclass[12pt]{amsart}
\usepackage{amstext}
\usepackage{amsthm}
\usepackage{amssymb}
\usepackage{amsmath}

\usepackage{enumitem}

\usepackage{graphicx}

\newtheorem{theorem}{Theorem}[section]
\newtheorem{lemma}[theorem]{Lemma}
\newtheorem{proposition}[theorem]{Proposition}
\newtheorem{question}[theorem]{Open Question}
\theoremstyle{definition}  
\newtheorem{definition}[theorem]{Definition}
\newtheorem{notation}[theorem]{Notation}
\theoremstyle{remark}  
\newtheorem*{remark}{Remark}

\title[]{Combinatorial Methods in Bootstrap Percolation Problems}
\author[G Nagy]{G\'abor V. Nagy}\address{Bolyai Institute, University of Szeged and John von Neumann University, Kecskemét}
\thanks{The research leading to these results has received funding from the national project TKP2021-NVA-09.
Project no.\ TKP2021-NVA-09 has been implemented with the support provided by the Ministry of Culture and Innovation of Hungary from the National Research,
Development and Innovation Fund, financed under the TKP2021-NVA funding scheme.}
\email{ngaba@math.u-szeged.hu}
\author[Á Süli]{Ákos Süli}\address{Bolyai Institute, University of Szeged}
\email{suliakos@yahoo.com}

\begin{document}

\begin{abstract}
An elegant bootstrap percolation result on the hyperrectangle graph was proved by Balogh, Bollobás, Morris, and Riordan using a linear algebric method in 2012.
This paper studies the same problem by purely combinatorial means. The base case is solved for all dimensions, and we pose open problems for the remaining cases.
\end{abstract}

\maketitle

%

\section{Introduction}

We examine the following process. We are given a hypergraph $\mathcal{H}$ and an initial infected vertex set $A_0 \subseteq V(\mathcal{H})$. For all $i \geq 1$, $A_i$ is defined as follows.
The vertex set $A_i$ is the union of $A_{i-1}$ and the set of vertices $v$ of $\mathcal{H}$ for which there exists an hyperedge $E$ in $\mathcal{H}$ for which $E \setminus A_{i-1} = \{v\}$ holds \cite{Balogh2012}.
We call such a hyperedge an \emph{infecting set of $v$} in $A_{i-1}$, and we also say that $E$ \emph{infects} $v$.

Notice that $(A_i)$ is a monotone increasing set sequence. We only interested in this process when $V(\mathcal{H})$ is finite. In this case, there will be a smallest index $f$ for which $A_f = A_{f+1}$.
The set $A_f$ is the \emph{full form} of $A_0$ which we denote by $\overline{A_0}$. If $\overline{A_0}=V(\mathcal{H})$, we say that $A_0$ \emph{percolates} in $\mathcal{H}$. We also call $A_0$ a \emph{percolating set} in $\mathcal{H}$.

We call the process described above \emph{percolation by phases}. Another, similar process can be defined. Again, we are given a finite hypergraph $\mathcal{H}$ and an initial infected vertex set $A^{(0)} \subseteq V(\mathcal{H})$.
For all $i \geq 1$, $A^{(i)}$ is defined as follows. Pick a vertex $v$ in $V(\mathcal{H}) \setminus A^{(i-1)}$ for which an infecting set exists in $A^{(i-1)}$ (as defined above), if such a vertex exists.
Then, we define $A^{(i)}$ as $A^{(i-1)} \cup \{v\}$. This process is called \emph{step-by-step percolation}. At some point during the iteration of this process starting from a set $A^{(0)}$, there will be an index $\ell$ such that no vertex has an infecting edge in $A^{(\ell)}$.
Then, the process terminates. Unlike percolation by phases, this process is generally not deterministic as vertex $v$ above is usually not unique.
The set sequence $A^{(i)}$, too, is monotone increasing. It is straightforward to prove that starting from any initial infected vertex set $A^{(0)}$, the terminating set $A^{(\ell)}$ is equal to $\overline{A^{(0)}}$ regardless of the vertices chosen in step-by-step percolation. We make heavy use of this fact in our analysis.

We restrict our attention to a special class of hypergraphs. These hypergraphs can be defined in the following way. 

\begin{notation}
	We denote the set $\{1,2,\ldots, n\}$ with $[n]$.
\end{notation}

\begin{definition}
	Fix positive integers $d, t, r$ and $n_1, n_2, \ldots, n_d$ such that $d \geq r \geq 1$ and $t \geq 2$. The \emph{hyperrectangle graph} corresponding to these parameters has the vertex set $[n_1] \times [n_2] \times \ldots \times [n_d]$.
	Furthermore, the hyperedges are the Cartesian products $I_1 \times I_2 \times \ldots \times I_d$ such that $I_i \subseteq [n_i]$ for all $i \in [d]$, and among the sets $I_i$, exactly $r$ have size $t$, and the rest have size $1$.
	(Roughly speaking, the hyperedges are the $r$-dimensional subhypercubes with side-length $t$.)

\end{definition}
If the value of $t$ and $r$ are clear from context, then by a slight abuse of notation, we also say that a set $A$ percolates in $[n_1] \times [n_2] \times \ldots \times [n_d]$ if it percolates in a hyperrectangle graph with parameters $d, t, r$ and $n_1, n_2, \ldots, n_d$. Note that the value of $d$ is implied by the vertex set.

Given a hyperrectangle graph and its parameters, we are interested in the minimal size of sets which percolate.
We denote this quantity with $m(\underline{n}, d, t, r)$ where $\underline{n}$ denotes the $d$-tuple $(n_1, n_2, \ldots, n_d)$.
Noga Alon determined  $m(\underline{n}, d, t, r)$ in the case $d=r$  using exterior algebra in 1985 \cite{Alon1985}.
In 2012, Balogh, Bollobás, Morris, and Riordan \cite{Balogh2012} established tight lower (as well as tight upper) bounds for an even more general class of hypergraphs than the one described above using linear algebraic methods. 

\begin{definition}\label{def:L}
	Let the values of the parameters $d, t, r$ and $n_1, n_2,\ldots,n_d$ be fixed. Let $L_{n_1,\ldots, n_d}^{d,t,r}$ denote the subset of elements of $[n_1] \times [n_2] \times \ldots \times [n_d]$ which has at most $r-1$ coordinates with value more than $t-1$.
\end{definition}

The set $L_{5,6}^{2,2,2}$ is visualized as described below on Figure \ref{fig:l_set}. It was showed in \cite{Balogh2012}, that $L_{n_1,\ldots, n_d}^{d,t,r}$ is a minimum-size percolating set in the hyperrectangle graph $[n_1] \times [n_2] \times \ldots \times [n_d]$ with parameters $d,t,r$:

\begin{theorem}[\cite{Balogh2012}, Theorem 4.]\label{theorem:BBMR}
	Let $d, t, r$ and $n_1, n_2, \ldots, n_d$ such that $d \geq r \geq 1$ and $t \geq 2$ be arbitrary positive integers. Then,
	\begin{equation*}
	m((n_1,n_2,\ldots,n_d), d, t, r) = |L_{n_1,\ldots, n_d}^{d,t,r}|=\sum_{s=0}^{r-1} \sum_{I \in \binom{[d]}{s}} (t-1)^{d-s} \prod_{i \in I} (n_i+1-t)
	\end{equation*}
	holds where $\binom{[d]}{s}$ denotes the set of $s$-element subsets of $[d]$.
\end{theorem}

\begin{figure}[htbp]
	\centering
	\includegraphics[width=0.4\linewidth]{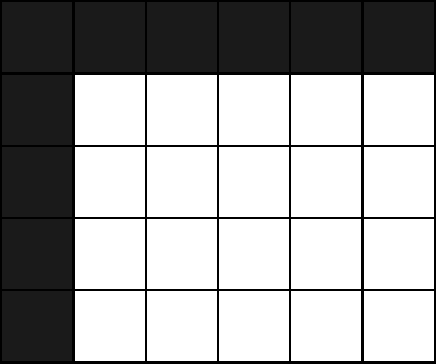}
	\caption {A visualization of the set $L_{5,6}^{2,2,2}$}
	\label{fig:l_set}
\end{figure}

We stated the theorem solely for hyperrectangle graphs, the original theorem is more general.
In this work, we study the above problem using only elementary, combinatorial methods. We give an elementary proof of the case $t=r=2$ for arbitrary dimension $d$,
and present some lemmas and ideas which may help to tackle the general case.
The combinatorial proofs in this work were all first presented in the second author's Bachelor's thesis \cite{Suli2024}.

\section{The Two-Dimensional Case}

We visualize the vertex set $[n_1] \times [n_2]$ of a hyperrectangle graph as a two-dimensional array of squares. We index the rows top to bottom in increasing order with the elements of $[n_1]$ and columns left to right in increasing order with the elements of $[n_2]$. If $r=2$ and $t=2$, a quadruple of squares or vertices form a hyperedge if and only if they are at the corners of an axis-aligned rectangle, see Figure \ref{fig:inf_set}. The white and light gray entries represent non-infected vertices while the dark gray entries represent infected vertices. The dark and light gray entries form an infecting set of the light gray entry. The light gray entry represents the vertex $(2,5)$.

\begin{figure}[htbp]
	\centering
	\includegraphics[width=0.4\linewidth]{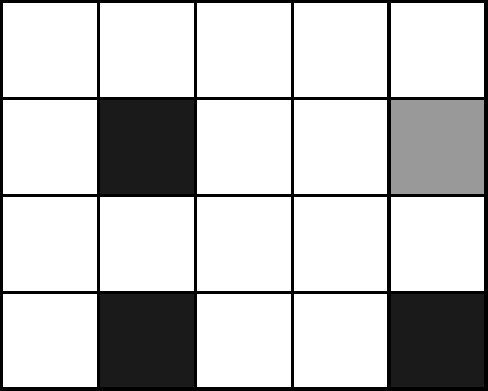}
	\caption{A visualization of an infecting set}
	\label{fig:inf_set}
\end{figure}

\begin{definition}\label{def:row}
	Fix $k \in [n_1]$ and a set $A\subseteq[n_1]\times[n_2]$. Let $A|_{(k,*)}$ be the set of elements of $A$ which have $k$ as their first coordinate.
	We call the set $A|_{(k,*)}$ the \emph{$k$th row} of $A$. The subset of $[n_2]$ containing the second coordinates of the elements of $A|_{(k,*)}$ is the \emph{$k$th $P$-row} of $A$
	(or the \emph{projection} of the $k$th row).
	The \emph{$j$th column} and the \emph{$j$th $P$-column} are defined analogously by switching the roles of $[n_1]$ and $[n_2]$, and of the first and second coordinates.
\end{definition}

First, we study the problem when $d=t=r=2$, and $n_1$ and $n_2$ are arbitrary. We present a combinatorial proof of this special case of Theorem \ref{theorem:BBMR}. 

\begin{proposition} \label{prop:2D}
	$m((n_1, n_2), 2,2,2) = n_1 + n_2 - 1$ for any $n_1, n_2$ positive integers.
\end{proposition}

Notice that due to the symmetry in the definition of hyperedges of the hyperrectangle graph, the property of whether a set percolates or not is invariant under row and column permutations.
Therefore, one can consider the multiset of $P$-rows of an initially infected vertex set only to determine if it percolates or not. We apply this observation in the following lemma which is crucial to our elementary proof.

\begin{definition}
	Let $r_1$ and $r_2$ be two rows of a set $A\subseteq[n_1]\times[n_2]$. The \emph{union} $r_1\cup r_2$ is a row whose projection is equal to the union of
	the projections of $r_1$ and $r_2$.
\end{definition}

\begin{lemma}\label{lemma:2Dunion}
	Let $A \subseteq [n_1] \times [n_2]$ be a percolating set, and $t=r=2$. Let $A^-$ be a set constructed by removing two rows of $A$ and adding their union as a new row. Then, the set $A^- \subseteq [n_1-1] \times [n_2]$ percolates.
\end{lemma}

\begin{proof}
	Let $K$ be an arbitrary subset of $[n_1] \times [n_2]$. Let $\mathcal{U}(K)$ be the set constructed by removing the $n_1$th row and replacing the ($n_1-1$)th row with the union of the (original) $n_1$th and $(n_1-1)$th rows of $K$. See Figure \ref{fig:union_op}.
	
	Without loss of generality, we can assume that $A^-$ is $\mathcal{U}(A)$.
	Consider a step-by-step percolation process of $A$. Let $\ell$ denote the smallest index for which $A^{(\ell)} = \overline{A}$.
	
	$$A=A^{(0)}, A^{(1)}, A^{(2)}, \ldots, A^{(\ell)}=[n_1]\times[n_2]$$	
	Let $K$ and $L$ be subsets of $[n_1] \times [n_2]$. We show that if $L$ can arise as the union of $K$ and a vertex which has an infecting subset in $K$, then either $\mathcal{U}(L)$ arises the same way from $\mathcal{U}(K)$ or $\mathcal{U}(L)=\mathcal{U}(K)$. From this, the proposition follows, since after removing repeating elements, the sequence 
	$$A^-=\mathcal{U}(A^{(0)}),\,\mathcal{U}(A^{(1)}),\,\mathcal{U}(A^{(2)}),\,\dots,\,\mathcal{U}(A^{(\ell)})=[n_1-1]\times[n_2]$$	
	is a step-by-step percolation process of $A^-$.
	
	We prove this by case analysis. Let $L=K\cup\{v\}$, where $v$ is the infected vertex with infecting set $E=I_1\times I_2$.
	
	\begin{itemize}
		\item If $I_1\subseteq[n_1-2]$, in other words, if $E$ does not intersect the rows of $K$ modified by $\mathcal{U}$, then $\mathcal{U}(L)=\mathcal{U}(K)\cup\{v\}$. Therefore, $v$ is infected by $E^*:=E$ in $\mathcal{U}(K)$.
		
		\item If $I_1=\{n_1-1,n\}$, then $\mathcal{U}(L)=\mathcal{U}(K)$ since the union of $(n_1-1)$th and $n_1$th rows coincide in $K$ and $L$, and the rest of the rows are identical in $K$ and $L$.
		
		\item Finally, if $|I_1\cap\{n_1-1,n_1\}|=1$, then we have two subcases. 
		
		\begin{itemize}
			\item If $v$ is not in the $(n_1-1)$th or $n_1$th rows, then $\mathcal{U}(L)=\mathcal{U}(K)\cup\{v\}$. Let $I_1^*=(I_1\cap[n-2])\cup\{n-1\}$. Then, the vertex $v$ is infected by $E^*:= I_1^*\times I_2$ in $\mathcal{U}(K)$.
			
			\item Else, $v$ is in the $(n_1-1)$th or $n_1$th row. In this case, let $v^*$ be the vertex is obtained from $v$ by setting its first coordinate to $(n_1-1)$. Then $\mathcal{U}(L) = \mathcal{U}(K) \cup \{v^*\}$. If $v^* \in \mathcal{U}(K)$, then $\mathcal{U}(L)=\mathcal{U}(K)$ holds. Otherwise, $v^*$ is infected by the infecting set $E^*:=I_1^*\times I_2$ in $\mathcal{U}(K)$ where $I_1^*$ is defined as in the previous subcase.

		\end{itemize}

	\end{itemize}
	
	We have exhausted all cases and the proof is complete.
\end{proof}

\begin{figure}[htbp]
	\centering
	\includegraphics[width=0.8\linewidth]{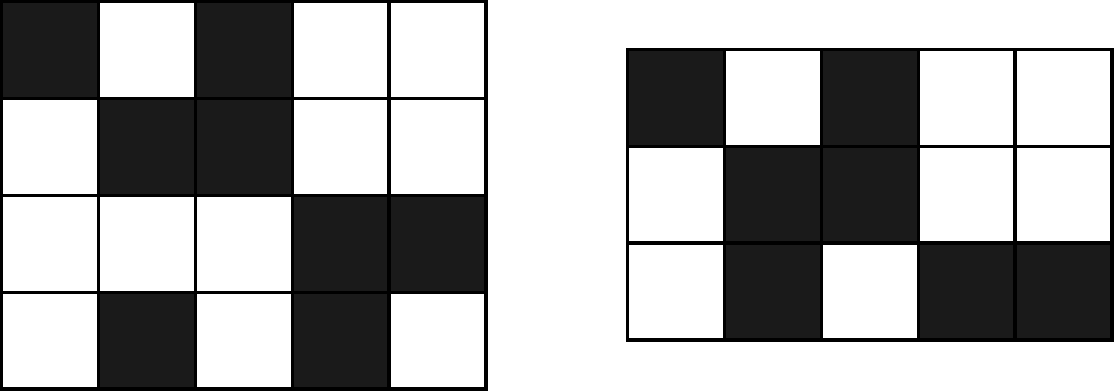}
	\caption {A visualization of a set (left) and its image under operation $\mathcal{U}$ (right)}
	\label{fig:union_op}
\end{figure}

\begin{proof}[Proof of Proposition \ref{prop:2D}]
	It is easy to check that the set $L := \{(x_1,x_2) \in [n_1] \times [n_2]: x_1 = 1 \text{ or } x_2=1\}$ is a percolating set in $[n_1] \times [n_2]$ with $n_1+n_2-1$ elements. It remains to show that $|A| \geq n_1+n_2-1$ for all percolating sets $A$.
	
	We prove this by induction on $n_1$. If $n_1 = 1$ then $A$ can not contain an infecting set. Therefore, it only percolates if and only if $A= \{1\} \times [n_2]$. The size of this set is $n_2=n_1+n_2-1$, thus the assertion holds.
	
	Assume that $n_1 \geq 2$ and the proposition holds for all subsets that percolate in $[n_1-1] \times [n_2]$. Consider the percolating set $A \subseteq [n_1] \times [n_2]$. Then, $A$ has two $P$-rows which are non-disjoint. This is because if $A \neq [n_1] \times [n_2]$, then the projections of the two rows containing the entries of any infecting set must have a non-empty intersection.
	Let $r_1$ and $r_2$ be two rows of $A$ whose projections are non-disjoint, and consider the set $A^-$ obtained by removing $r_1$ and $r_2$ and adding $r_1 \cup r_2$. Due to Lemma \ref{lemma:2Dunion}, the set $A^-$ also percolates, and since the rows $r_1$ and $r_2$ are non-disjoint, we have $|A| \geq |A^-| + 1$. Then, by the induction hypothesis,
	$$|A| \geq |A^-| + 1 \geq (n_1-1+n_2-1) + 1 = n_1+n_2-1$$
	holds and this completes the proof.
\end{proof}

Next, we show a combinatorial property of bootstrap percolation in the two-dimensional case. 

\begin{proposition}\label{prop:removal}
	Let $A \subseteq [n_1] \times [n_2]$ be a percolating set with $d=r=2$ and an arbitrary value of $t$. If a row (or column) containing exactly $t-1$ elements is removed from $A$, the resulting set $A^-$ still percolates in $[n_1-1] \times [n_2]$ (or in $[n_1] \times [n_2-1]$).
\end{proposition}

\begin{proof}[Proof]
	Fix an arbitrary $t\ge2$. Without loss of generality we assume that the last ($n_1$th) row of $A$ contains $t-1$ elements and $A^-$ is obtained from $A$ by the removal of the last row.
	The notation $B^-$ will be used analogously for arbitrary subsets $B \subseteq [n_1] \times [n_2]$.
	Throughout this proof, we employ the following notations for an
	\begin{equation*}A=A^{(0)}, A^{(1)}, A^{(2)}, \ldots, A^{(\ell)}=[n_1]\times[n_2]\tag{$*$}\end{equation*}
	step-by-step percolation process of $A$: For $k=1,\dots,\ell$, let $v_k$ denote the vertex that gets infected in the $k$th step, i.e.\ $A^{(k)}=A^{(k-1)}\cup\{v_k\}$,
	and let $E_k$ denote an infecting set of $v_k$ in $A^{(k-1)}$. Since $E_k$ is not unique (there can be more than one infecting set for a vertex), in this proof we always associate a fixed
	sequence $E_1,\dots,E_{\ell}$ of infecting sets to a step-by-step percolation process in ($*$) to be concrete.
	
	We will construct a special step-by-step percolation process ($*$) with infecting sets $E_1,\dots,E_{\ell}$ such that
	whenever the infected element $v_k$ is not in the last row, its infecting set $E_k$ avoids the last row (that is, $E_k$ has the form $I_1\times I_2$ for some $I_1\subseteq[n_1-1]$). We call such a process
	\emph{perfect}.
	It is obvious that for a perfect process ($*$), the sequence
	$$A^-=\left(A^{(0)}\right)^-, \left(A^{(1)}\right)^-, \left(A^{(2)}\right)^-, \ldots, \left(A^{(\ell)}\right)^-=[n_1-1]\times[n_2]$$
	is a step-by-step percolation process of $A^{-}$ after removing the repeating elements (when an element in the last row gets infected in the original process), which proves the proposition.
	
	First we construct a step-by-step percolation process ($*$) which satisfies a more technical condition. Let the \emph{hitting time} of a process ($*$) be the smallest index $h\in[\ell]$ 
	for which the infected vertex $v_h$ (of the $h$th step) is in the last row. We say that the process ($*$) is \emph{good} if there exists a $t-1$ element set $S\subseteq[n_1-1]$
	of row indices such that \begin{equation*}r_{n_1}\subseteq \bigcap_{s\in S} r_s\tag{$**$}\end{equation*}
	holds for all sets $A^{(k)}$ with $k\ge h$ (where $h$ is the hitting time of the process), and $r_s\subseteq[n_2]$ is the $s$th $P$-row of $A^{(k)}$ in the sense of Definition \ref{def:row}.
	(There is a slight abuse of notation here; we should write $r^{(k)}_s$ instead of $r_s$.)
	In fact, every good process ($*$) with infecting sets $E_1,\dots,E_{\ell}$ can also be made perfect by constructing a last-row-avoiding infecting set $E'_k$ for $v_k$ if $v_k$ is not in the last row:
	Before the hitting time $h$, every infecting set $E_k$ (where $k<h$) automatically avoids the last row, because there are too few (namely, $t-1$) elements in the last row to be contained in an infecting set.
	So we can set $E'_k:=E_k$ if $k<h$. Now consider an infecting set $E_k$ of a vertex $v_k$ not in the last row, for $k\ge h$. If $E_k$ avoids the last row, then there is nothing to do ($E'_k:=E_k$).
	Next assume that $E_k=I_1\times I_2$ does not avoid the last row, i.e.\ $n_1\in I_1$. Then by property~($**$), $v_k$ must be in a row whose index is not in $S$ either, and so there exist an index
	$s\in S\setminus I_1$. (Recall that $|I_1|=|S\cup\{n_1\}|=t$.) Property~($**$) implies that $E'_k:=((I_1\setminus\{n_1\})\cup\{s\})\times I_2$ is a last-row-avoiding infecting set of $v_k$ as desired.
	
	Now we are left to construct a good process ($*$). We will construct the sets $A^{(0)},\dots,A^{(\ell)}$ one by one following the algorithm described in the Introduction section. Set $A^{(0)}:=A$ and $\ell:=n_1n_2-|A|$.
	For $i=1,\dots,\ell$, let $N_i$ be the set of those vertices in $([n_1]\times[n_2])\setminus A^{(i-1)}$ for which an infecting set exists in $A^{(i-1)}$, let $v_i$
	be an arbitrary element of $N_i$, and set $A^{(i)}=A^{(i-1)}\cup\{v_i\}$. The fact that $A$ percolates guarantees that the set $N_i$ is always nonempty ($i=1,\dots,\ell$) regardless of the choice $v_j$'s, and the final set is always $A^{(\ell)}=[n_1]\times[n_2]$. So we have a freedom in the choice of $v_i$'s. Let us start the algorithm without any requirements on the $v_i$'s until we pick a vertex $v_h$ from the last row for some (smallest) index $i=h$. (This index $h$ will be the hitting time of the process just being created.) Let $E=I_1\times I_2$ be the infecting set of $v_h$ in $A^{(h-1)}$, and set $S:=I_1\setminus\{n_1\}$ for the goodness. The set $A^{(h)}:=A^{(h-1)}\cup\{v_h\}$ has been just created cleary satisfies ($**$), implied by the definition of infecting sets and the fact that $A^{(h-1)}$ has only $t-1$ elements in its last row (so all of them must be contained in $E$). From now on we will be more careful when selecting $v_i$. For $i=h+1,\dots,\ell$, we only pick an element $v$ from $N_i$
	(to set $v_i:=v$) if it will not violate condition ($**$): the selection of exactly those vertices $v$ are forbidden which are in the last row of $[n_1]\times[n_2]$ and for which $S\times\{\text{col}(v)\}\not\subseteq A^{(i-1)}$, where $\text{col}(v)$ denotes the column index of $v$. The point is that $N_i$ always contains at least one vertex $v'$ whose selection is permitted. This is because if $N_i$
	contains a vertex $v$ (in the last row) whose selection is forbidden, then all vertices of $S\times\{\text{col}(v)\}$ which are not in $A^{(i-1)}$ must be contained in $N_i$ (whose selection is permitted). To see this,
	let $u\in(S\times\{\text{col}(v)\})\setminus A^{(i-1)}$ be arbitrary. The fact $v\in N_i$ means that there exists an infecting set $E_v=I_1\times I_2$ of $v$ in $A^{(i-1)}$ with $n_1\in I_1$.
	Condition ($**$) for $A^{(i-1)}$ implies that $n_1$ can be replaced to the row index of $u$ in $I_1$ to obtain an infecting set $I'_1\times I_2$ of $u$ in $A^{(i-1)}$, that is, $u\in N_i$. Hence we verified that it is possible to pick a vertex $v_i\in N_i$ such that $A^{(i)}:=A^{(i-1)}\cup\{v_i\}$ satisfies condition ($**$). The proof is now complete.
\end{proof}

\begin{remark}
	This proposition yields another proof of the lower bound of Proposition \ref{prop:2D} via induction on $n_1+n_2$. Let $A$ be a percolating set. The statement trivially holds in $n_1=n_2=1$. Due to the symmetry of $n_1$ and $n_2$ in the definition of hyperedges, one can assume that $n_1 \geq n_2$. If every row of $A$ has at least two elements, the statement holds in this case, as $|A| \geq 2n_1 \geq n_1+n_2-1$. Otherwise, $A$ has at least one row with at most one element.
	These rows contain exactly one element since otherwise $A$ would not percolate since an empty row remains empty throughout the whole percolation process. After the removal of such a row, applying the Proposition \ref{prop:removal} with $t-1=1$ and the induction hypothesis finishes the proof.
\end{remark}

As of now, we are unaware of combinatorial proofs for values of $t$ greater than $2$. However, we can prove the tight lower bound under a special condition defined as follows.

\begin{definition}
	Let $A$ be a subset of $V(\mathcal{H})$ where $\mathcal{H}$ is an arbitrary hypergraph. The set $A$ \emph{percolates in one phase} if for all vertices in $V(\mathcal{H}) \setminus A$, there exists an infecting set in $A$.
\end{definition}

Note that in the definition of percolation by phases, by setting $A_0 = A$, this is equivalent to $A_1 = V(\mathcal{H})$. The set depicted on Figure \ref{fig:l_set} is an example of a set which percolates in one phase. We give an elementary proof of Theorem \ref{theorem:BBMR} in this special case.

\begin{theorem} \label{theorem:gen_t}
	Let $d=r=2$ and let $t \geq 2$ be arbitrary. Consider an initially infected vertex set $A \subseteq [n_1] \times [n_2]$ with $n_1, n_2 \geq t$. If $A$ percolates in one phase, then its size is at least $(n_1+n_2) \cdot (t-1) - (t-1)^2$.
\end{theorem}

\begin{proof}
	We prove this by induction on $n_2$. In the base case, we have $n_2=t$. It is easy to see that if $A \subseteq [n_1] \times [t]$ percolates in one phase, then $A$ must have at least $t-1$ rows with $t$ elements, and the rest of the rows of $A$ must have at least $t-1$ elements. Therefore, $A$ has at least
	$$(t-1) \cdot t + (n_1-(t-1)) \cdot (t-1) = (n_1+n_2) \cdot (t-1) - (t-1)^2$$
	elements, proving the assertion.
	
	Now, let $A$ be a set that percolates in one phase in $[n_1] \times [n_2]$ where $n_2 > t$. The first column of $A$ must contain at least $t-1$ elements, since otherwise no infecting set could exist for the non-infected vertices in the first column. Now, consider the set $A^-$ obtained from $A$ by removing its first column. We have two cases in the induction step. 
	
	\vspace{0.2cm}
	
	\textit{First case.} Assume that $A^-$ percolates in one phase in $[n_1] \times ([n_2] \setminus \{1\})$. Then, $|A| \geq |A^-| + (t-1)$, and the induction hypothesis yields
	$$|A| \geq |A^-| + (t-1) \geq (n_1+n_2-1) \cdot (t-1) - (t-1)^2 + (t-1) = (n_1+n_2) \cdot (t-1) - (t-1)^2.$$
	This verifies the theorem in this case. 
	
	\vspace{0.2cm}
	
	\textit{Second case.} Assume that $A^-$ does not percolate in one phase. Our aim is to construct a set $A^*$ for which the following requirements hold. 
	
	\begin{enumerate}[label=(\Roman*)]
		\item The size of $A^*$ is the same as the size of $A$. \label{reqI}
		
		\item The set $A^*$ has at least $t-1$ elements in its first column. \label{reqII}
		
		\item By removing the first column of $A^*$, the resulting  set $(A^*)^-$ percolates in one phase in $[n_1]  \times ([n_2] \setminus \{ 1 \})$. \label{reqIII}
	\end{enumerate}
	Requirements \ref{reqII} and \ref{reqIII} ensure that the calculations of the first case can be applied to $A^*$, which give the desired lower bound for $|A|$ as well since $|A| = |A^*|$ by requirement \ref{reqI}.
	
	Since $A^-$ does not percolate in one phase, there exists a non-infected vertex $v \in [n_1]  \times ([n_2] \setminus \{ 1 \})$ which only has an infecting set $E_0 = I_1 \times I_2$ with $1 \in I_1$. We may assume without loss of generality that $v = (t,t)$ and that one of its infecting sets is $E_0 = [t] \times [t]$. 
	
	We construct $A^*$ from $A$ row by row in the following way. Let $r_k \subseteq [n_2]$ denote the $k$th $P$-row of $A$ in the sense of Definition \ref{def:row}.
	If $1 \in r_k$ and $|r_k \cap [t]| < t$, then let the $k$th $P$-row of $A^*$ be $r_k \setminus \{1\} \cup \{\min ([t] \setminus r_k) \}$. Less formally and more visually, this means that in this case, we construct the $k$th row of $A^*$ from the $k$th row of $A$ by moving the first element of the row to the leftmost empty place in that row. If these two conditions do not hold for a $P$-row $r_k$ of $A$, we let the $k$th $P$-row of $A^*$ be $r_k$, in other words, we leave the row unchanged. See Figure \ref{fig:constr} for an example of this construction with $t=2$.

	Now, we verify that the requirements \ref{reqI}, \ref{reqII}, and \ref{reqIII} hold for $A^*$. Since $A^*$ is constructed in such a way that each of its rows have the same number of elements as the row of the same index in $A$, requirement \ref{reqI} is satisfied.
	By the construction of $A^*$, the edge $E_0$ defined above guarantees that $[t-1] \times \{1\} \subseteq A^*$, implying requirement \ref{reqII}.
	
	Finally, we prove the that requirement \ref{reqIII} holds. Let $v \notin (A^*)^-$ denote an arbitrary non-infected vertex in $[n_1] \times ([n_2] \setminus \{1\})$. Since $A^- \subseteq (A^*)^-$ by the construction of $A^*$, if $v$ has an infecting set in $A^-$, then $v$ also has an infecting set in $(A^*)^-$. Thus, we can assume that $v$ has no infecting set in $A^-$, and we need to verify that $v$ has an infecting set in $(A^*)^-$. Since $A$ percolates in one phase, $v$ has an infecting set $E = I_1 \times I_2$ in $A$. As vertex $v$ does not have an infecting set in $A^-$, $1 \in I_2$ must hold. Since $E$ is an infecting set of a vertex outside the first column, $E$ must have $t$ elements in the first column of $A$. Denote the index of the column $v$ belongs to by $k_v \geq 2$. Let $k^* = \min \big([t] \setminus (I_2 \setminus \{k_v\})\big)$ which is well-defined since $|I_2 \setminus \{k_v\}| = t-1$. The minimality of $k^*$ implies that $I_1 \times [k^*-1] \subseteq A$, therefore $I_1 \times \{k^*\} \subseteq A^*$ by the construction of $A^*$. Thus, since $v \notin A^*$, the vertex $v$ can not be in the column indexed $k^*$, i.e. $k^* \neq k_v$. Therefore, $k^* \notin I_2$, that is, $E$ avoids the $k^*$th column of $[n_1] \times [n_2]$. Now, consider the hyperedge $E^* := I_1 \times (I_2 \setminus \{1\} \cup \{k^*\})$. Since $I_1 \times \{k^*\} \subseteq A^*$, $A^- \subseteq (A^*)^-$, and $E \setminus \{v\} \subseteq A$, the set $E^*$ is an infecting set of $v$. The choice of $v$ was arbitrary, therefore it follows that $(A^*)^-$ percolates in one phase.

	With this, all requirements are verified, therefore the calculations in the first case can indeed be applied to $A^*$. This finishes the proof.
\end{proof}

\begin{figure}[htbp]
	\centering
	\includegraphics[width=0.7\linewidth]{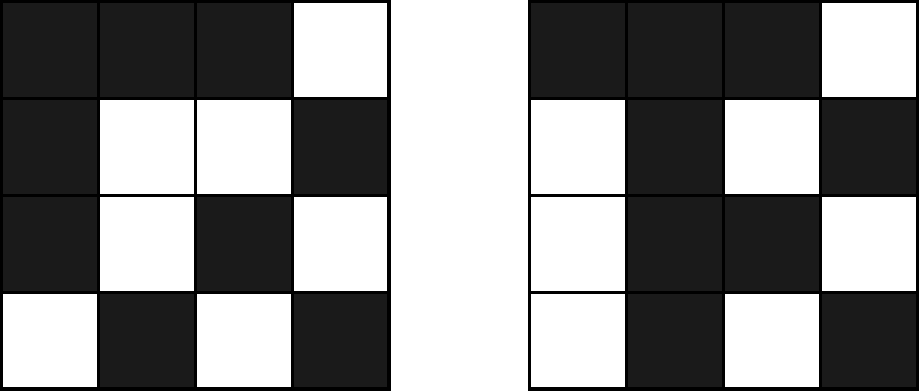}
	\caption{An example of a set $A$ percolating in one phase with $t=2$ (left) and its corresponding set $A^*$ (right)}
	\label{fig:constr}
\end{figure}

\section{A Proof for General Number of Dimensions}

In this section, we generalize the proof provided for Proposition \ref{prop:2D} to any number of dimensions. We do this by generalizing the definitions and the lemma introduced in the previous section. Then, we provide an elementary proof of the following proposition.

\begin{theorem} \label{theorem:gen_D}
	Let $d \geq 2$ and $n_1, n_2, \ldots, n_d$ be arbitrary positive integers. Then,
	$$m((n_1,n_2,\ldots,n_d),d,2,2)= \sum_{i=1}^{d} n_i - (d-1).$$
\end{theorem}

Now we define the higher dimensional analogues of rows and columns.

\begin{definition}
	Let $A \subseteq [n_1] \times [n_2] \times \ldots \times [n_d]$ and let $k \in [d]$. Define the following mapping:
	$$\Pi_k: \mathbb{Z}^d \rightarrow \mathbb{Z}^{d-1}, (x_1, x_2, \ldots, x_d) \mapsto (x_1, x_2, \ldots, x_{k-1}, x_{k+1}, \ldots, x_d).$$
	Let $A^m_k$ denote the set of elements of $A$ which have $m$ as their $k$th coordinate. We call $A^m_k$ the \emph{the $m$th slice of $A$ along the $k$th coordinate},
	and we call $\Pi_k(A^m_k)$ the projection of $A^m_k$ or \emph{the $m$th $P$-slice of $A$ along the $k$th coordinate}.
	The union of slices $s_1$ and $s_2$ along the $k$th coordinate is a slice along the $k$th coordinate whose projection is equal to the union of projections of $s_1$ and $s_2$.
\end{definition}

\begin{remark}
	The rows and columns of a two-dimensional set are the slices of this set according to their first and second coordinate, respectively.
	Similarly to the two-dimensional case, whether or a set percolates or not remains invariant under permutation of the slices along any coordinate. 
\end{remark}

\begin{lemma}\label{lemma:gen_Dunion}
	Let $A \subseteq [n_1] \times [n_2] \times \ldots \times [n_d]$ be a percolating set.
	Let $A^-$ be a set constructed by removing two slices along the same coordinate of $A$ and adding their union as a new slice. Then, $A^-$ percolates.
\end{lemma}

\begin{proof}[Sketch of proof]
	We take a very similar approach to proving this as in the proof of Lemma \ref{lemma:2Dunion}. Fix $k \in [d]$. Let $\mathcal{U}$ denote the operation that maps a set \linebreak $K \subseteq [n_1] \times [n_2] \times \ldots \times [n_d]$ to the set constructed by removing the $n_k$th and ($n_k-1$)th slices of $K$ along the $k$th coordinate, and adding their union as the ($n_k-1$)th slice along the $k$th coordinate. We may assume without loss of generality that $A^- = \mathcal{U}(A)$. We want to show the following. Apply the operation $\mathcal{U}$ to each element of a step-by-step percolation process of $A$. After removing the repeating elements, we obtain a step-by-step percolation process of $A^-$.
	
	First, we examine what happens when the vertex infected within a set in a step-by-step percolation process can be infected using an infecting set entirely contained by a slice along the $k$th coordinate. In this case, the same infecting set is present in the image of the set with respect to $\mathcal{U}$, or the vertex infected by it is already in the image, possibly by changing the $k$th coordinate to $n_k-1$ in both cases. If the vertex can not be infected with an infecting set entirely contained in a slice along the $k$th coordinate, observe the following. The infecting set is contained within a subset (a ``two-dimensional plane'') of $[n_1] \times [n_2] \times \ldots \times [n_d]$ in which only two coordinates take on more than one value, one of which is the $k$th coordinate. By omitting the rest of coordinates, we return to the two-dimensional case. Here, Lemma \ref{lemma:2Dunion} finishes the proof if the $k$th coordinate takes on the role of the first coordinate.
\end{proof}

\begin{definition}
	The \emph{edgesum} of the product $[n_1] \times [n_2] \times \ldots \times [n_d]$ is $\sum_{i=1}^d n_i$.
\end{definition}

\begin{proof}[Proof of Theorem \ref{theorem:gen_D}]
	In \cite{Balogh2012}, a percolating set of size $\sum_{i=1}^{d} n_i - (d-1)$ is presented in the hyperrectangle graph $[n_1] \times [n_2] \times \ldots \times [n_d]$, for $t=r=2$. (This is the set $L_{n_1,\ldots, n_d}^{d,2,2}$ defined in Definition \ref{def:L}.) It remains to show that if $A$ is a percolating set in the hyperrectangle graph $[n_1] \times [n_2] \times \ldots \times [n_d]$, then $|A| \geq \sum_{i=1}^{d} n_1 - (d-1)$.
	
	The proof is by induction on the edgesum of the vertex set of the hyperrectangle graph. In the base case, $n_i=1$ for all $i \in [d]$, thus the edgesum is $d$. In this case, a set percolates if and only if it has one element, therefore the proposition holds.
	
	Assume that the proposition holds if the edgesum is $N-1$. Now, let $A$ be a percolating set with edgesum $N \geq d+1$.
	If $A \neq [n_1] \times [n_2] \times \ldots \times [n_d]$ -- otherwise, the theorem holds, -- $A$ has to contain an infecting set $E$.
	Let $k$ be a coordinate in which two elements of the infecting set differ. Then, the slices along the $k$th coordinate which contain the four elements of $E$ have non-disjoint projections.

	According to Lemma \ref{lemma:gen_Dunion}, the set $A^-$  obtained from $A$ by replacing these two slices with their union still percolates. Since these $P$-slices are not disjoint, $|A| \geq |A^-| + 1$ holds. The induction hypothesis yields that the size of $A$ is at least $N-1-(d-1)+1=N-(d-1)$, and this finishes the proof.
\end{proof}

\section{Open Questions}

We conclude by posing a few open questions relating to the combinatorial investigation of bootstrap percolation.


In \cite{Balogh2012}, it is proved that $L_{n_1,\ldots, n_d}^{d,t,r}$ percolates, and that
$$|L_{n_1,\ldots, n_d}^{d,t,r}| = m((n_1,n_2, \ldots, n_d), d, t, r)$$
holds for any values of $d, t, r$ and $n_1,n_2,\ldots, n_d$. Therefore, not only does $L_{n_1,\ldots, n_d}^{d,t,r}$ percolate in the hyperrectangle graph with vertex set $[n_1] \times [n_2] \times \ldots \times [n_d]$ and parameters $d, t, r$ but it does so with the least possible number of initially infected vertices.

\begin{definition}
	Let $\mathcal{H}$ be a hyperrectangle graph with arbitrary $d, t,$ and $r$ parameters, and let $A$ be an arbitrary subset of $V(\mathcal{H})$. Let $E$ be a hyperedge of the hypergraph which is an infecting set of some vertex $v$. In other words, let $E \in E(\mathcal{H})$ such that $|E \cap A| = t^r-1$ and $E \setminus A = \{v\}$. Now, let $w \in E \setminus \{v\}$ arbitrary. The \emph{shift operation} on $A$ with respect to $E$ and $w$ maps $A$ to $A \cup \{v\} \setminus \{w\}$. If $w$ is the vertex of $E$ with the maximal coordinate sum, this operation is a \emph{maximal shift operation}. We say that a set $A$ can be \emph{transformed into $A^*$ with a sequence of (maximal) shift operations} if there is a sequence of sets of which $A$ is the first element, and every other element can be obtained by a (maximal) shift operation on the previous element, and this sequence contains $A^*$.
\end{definition}

\begin{definition}
	Let $\mathcal{H}$ be a hyperrectangle graph and let $A \subseteq V(\mathcal{H})$ be arbitrary. The infecting set $E \in E(\mathcal{H})$ is in \emph{standard position} if the vertex $v \in E \setminus A$ has the maximal coordinate sum among the vertices in $E$.
\end{definition}

Note that a maximal shift operation can not be performed on a hyperedge if and only if it is in standard position.
An example of a maximal shift operation with $d=r=2$ and $t=3$ can be seen on Figure \ref{fig:im_op}. The dark gray entries are the infected vertices which along with the light gray entry form an infecting set of the light gray entry.

\begin{figure}[htbp]
	\centering
	\includegraphics[width=0.7\linewidth]{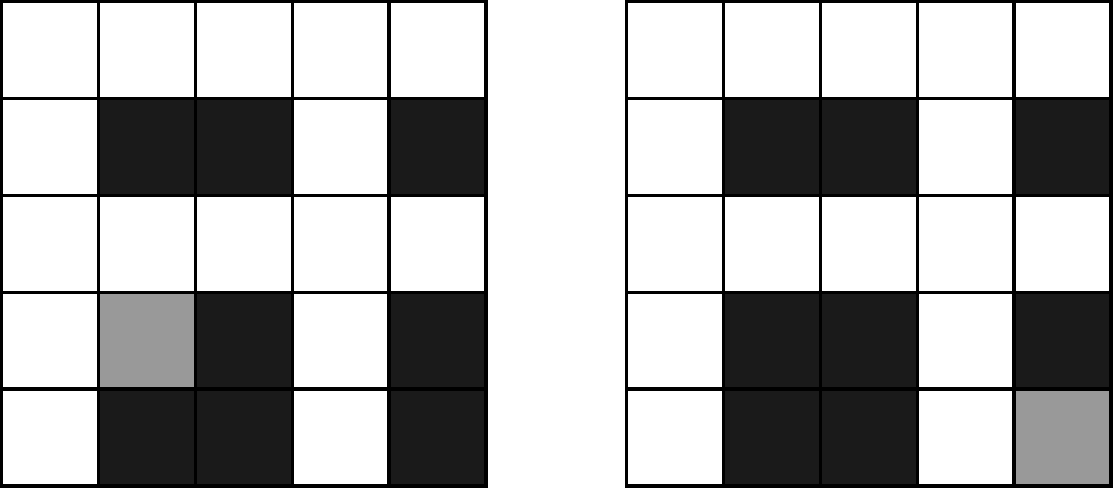}
	\caption{An infecting set (left) and the result of a maximal shift operation performed with respect to it (right)}
	\label{fig:im_op}
\end{figure}

Let $\mathcal{H}$ be a hyperrectangle graph, let $A \subseteq V(\mathcal{H})$ be arbitrary. Let $E \in E(\mathcal{H})$ be an arbitrary infecting set with $v \in E \setminus A$ and $w \in E \setminus \{v\}$. Let $A^*$ be the image of a shift operation on $A$ with respect to $E$ and $w$. Observe a step-by-step percolation process of $A$ and $A^*$. First, let $E$ infect $v$ in $A$, and let $E$ infect $w$ in $A^*$. Notice that we get the same set after this step. Therefore, $\overline{A} = \overline{A^*}$ holds. Consequently, transforming a percolating set using a sequence of shift operations yields a percolating set of the same size.

Since $L_{n_1,\ldots, n_d}^{d,t,r}$ is percolating set with the minimal possible number of infected vertices, we raise the following question.

\begin{question}\label{question:general}
	For what values of $d, t,$ and $r$ does it hold that any percolating set $A \subseteq [n_1] \times [n_2] \times \ldots \times [n_d]$ can be transformed into $A'$ using a sequence of shift operations such that $L_{n_1,\ldots, n_d}^{d,t,r} \subseteq A'$? Can this be done using a sequence of maximal shift operations only?
\end{question}

Note that verifying this for specific values of $d,t,$ and $r$ yields the tight lower bound for $m(\underline{n},d,t,r)$ with these parameters. We also raise the following question which relaxes the assumptions posed in the previous question.

\begin{question}
	Assume that $d=r$. For what values of $d$ and $t$ does there exist a sequence of shift operations for any percolating set $A \subseteq [n_1] \times [n_2] \times \ldots \times [n_d]$ which transforms $A$ into a set which percolates in one phase?
\end{question}

Note that an affirmitive answer in the case of $d=2$ would give a proof for the tight lower bound using Theorem \ref{theorem:gen_t} for any value of $t$. This was the main motivation behind Theorem \ref{theorem:gen_t}.
We finish this work by answering both questions in Open Question \ref{question:general} in the affirmative in the case of $d=t=r=2$.

\begin{proposition}
	Let $A \subseteq [n_1] \times [n_2]$ be a percolating set. Then, $A$ can be transformed using a sequence of maximal shift transformations into a set which has $n_2$ elements in its first row and $n_1$ elements in its first column.
\end{proposition}

\begin{proof}
	Let $A$ be a percolating set. Perform a sequence of maximal shift operations of $A$ until it is no longer possible to do so. Since the sum of the coordinates in the set strictly decreases sequence after performing a maximal shift operation, this sequence is finite for any percolating set. Let $A^*$ denote the last element of this sequence.
	
	We now define a relation over the $P$-rows of $A^*$. Let $r_1$ and $r_2$ be $P$-rows of $A^*$, and let $r_1 \sim r_2$ if and only if they are non-disjoint. Since all $P$-rows of $A^*$ are non-empty, the relation $\sim$ inherits reflexivity and symmetry from set intersection. Over the $P$-rows of $A^*$, this relation is also transitive. Let $r_i, r_j,$ and $r_k$ be $i$th, $j$th, and $k$th $P$-rows of $A^*$ respectively, and let $r_i \sim r_j$ and $r_j \sim r_k$. Since all infecting sets are in standard position in $A^*$, $\min r_i = \min r_j = \min (r_i \cap r_j)$ holds. For analogous reasons, $\min r_j = \min r_k = \min (r_j \cap r_k)$ also holds. Therefore, $\min r_i = \min r_k$, thus $r_i$ and $r_k$ are not disjoint, proving that $\sim$ is indeed transitive over the $P$-rows of $A^*$.
	
	Let $C_1, C_2, \ldots, C_s$ denote the equivalence classes of the $P$-rows of $A^*$ with respect to $\sim$, and for $i = 1,2,\ldots, s$, let $\widehat{r}_i \subseteq [n_2]$ be the element of $C_i$ with the lowest row index. Notice that $\widehat{r}_i$ contains the union of all other $P$-rows in its equivalence class. Otherwise, there would exist an infecting set in $A^*$ not in standard position. The $P$-rows $\widehat{r}_1, \widehat{r}_2, \ldots, \widehat{r}_s$ are pairwise disjoint since they are in pairwise different equivalence classes. From the above, it follows easily that 
	\begin{equation}\label{teljes_alak_kar}\overline{A^*} = \bigcup_{i=1}^{s} (C_i^{\text{ind}} \times \widehat{r}_i)\end{equation}
	where $C_i^{\text{ind}} \subseteq [n_1]$ is the set of row indices of the $P$-rows in $C_i$.
	
	So the set $A^*$ percolates if and only if there is only one equivalence class with $C_1^{\text{ind}} = [n_1]$ and $\widehat{r}_1 = [n_2]$. Therefore, the first row of $A^*$ contains $n_2$ elements. Since all elements of this equivalence class share the same minimum and $\min \widehat{r}_1 = \min [n_2] = 1$, every other $P$-row must contain the element $1$. Therefore, the first column of $A^*$ contains $n_1$ elements. This completes the proof.
\end{proof}

\begin{remark}
The above proof shows that the full form of a set $A \subseteq [n_1] \times [n_2]$ always has the form
$$(I_1\times J_1)\cup (I_2\times J_2) \cup\dots\cup (I_s\times J_s)$$
for some nonnegative integer $s$ and pairwise disjoint sets $I_1,I_2,\ldots,I_s\subseteq [n_1]$ and pairwise disjoint sets $J_1,J_2,\ldots,J_s\subseteq [n_2]$, c.f.\ formula~(\ref{teljes_alak_kar}).
\end{remark}

\bibliographystyle{plain}
\bibliography{references} 

\end{document}